\documentclass[10pt]{article}

\usepackage{a4wide}
\usepackage{amsmath,amssymb,amsthm}
\usepackage{url}

\usepackage{cite}

\makeatletter
\def\@biblabel#1{#1.}
\makeatother

\newtheorem{thm}{Theorem}[section]

\newtheorem{lem}{Lemma}[section]

\theoremstyle{remark}

\newtheorem{defin}{Definition}[section]
\newtheorem{rem}{Remark}[section]
\newtheorem{ex}{Example}[section]

\begin{document}

\title{The Hahn Quantum Variational Calculus}

\author{A. B. Malinowska$^{1}$
\and
D. F. M. Torres$^{2}$}

\date{{\small (Submitted to JOTA: 3/March/2010; 4th revision: 9/June/2010; accepted: 18/June/2010)}}

{\renewcommand{\thefootnote}{}
\footnotetext{This work was partially supported by the \emph{Portuguese Foundation
for Science and Technology} (FCT) through the \emph{Center for Research
and Development in Mathematics and Applications} (CIDMA).
The first author is currently a researcher at the University of Aveiro with the support
of Bia{\l}ystok University of Technology, via a project of
the Polish Ministry of Science and Higher Education ``Wsparcie
miedzynarodowej mobilnosci naukowcow''.
The authors are grateful to Nat\'alia Martins
for reading a preliminary version of the
paper, and for many useful remarks.}}

\footnotetext[1]{Assistant Professor
in the Faculty of Computer Science,
Bia{\l}ystok University of Technology,
15-351 Bia\l ystok, Poland. E-mail: \texttt{abmalinowska@ua.pt}}

\footnotetext[2]{Corresponding author.
Associate Professor in the Department of Mathematics,
University of Aveiro, 3810-193 Aveiro, Portugal.
E-mail: \texttt{delfim@ua.pt}}

% ---------------------------------------------------

\maketitle

% ---------------------------------------------------

\begin{abstract}
We introduce the Hahn quantum variational calculus. Necessary and
sufficient optimality conditions for the basic, isoperimetric, and
Hahn quantum Lagrange problems, are studied. We also show the
validity of Leitmann's direct method for the Hahn quantum
variational calculus, and give explicit solutions to some concrete
problems. To illustrate the results, we provide several examples and
discuss a quantum version of the well known Ramsey model of
economics.

\bigskip


\bigskip

\end{abstract}

% ---------------------------------------------------

\section{Introduction}

Quantum difference operators are receiving an increase of interest
due to their applications --- see, \textrm{e.g.},
\cite{Almeida,Bang04,Bangerezako,Cresson,Kac}.
Roughly speaking, a quantum calculus substitute the classical derivative
by a difference operator, which allows to deal with sets of nondifferentiable functions.
In \cite{Hahn}, Hahn introduced the quantum difference operator $D_{q,\omega}$,
where $q\in]0,1[$ and $\omega>0$ are fixed. The Hahn operator
unifies (in the limit) the two most well known and used
quantum difference operators: the Jackson $q$-difference derivative $D_{q}$,
where $q\in ]0,1[$ (\textrm{cf.} \cite{Gasper,Jackson1,Kac});
and the forward difference $\Delta_{\omega}$, where $\omega>0$
(\textrm{cf.} \cite{Bird,Jagerman,Jordan}).
The Hahn difference operator is a successful tool for constructing
families of orthogonal polynomials and investigating some
approximation problems --- see, \textrm{e.g.},
\cite{alvares,costas,odzi,Kwon,Petronilho}. However, only in 2009
the construction of a proper inverse of $D_{q,\omega}$ and
the associated integral calculus was given \cite{Aldwoah,Annaby}.

In this work we introduce the variational Hahn calculus. More
precisely, we formulate problems of the calculus of variations using
Hahn's difference operator and the Jackson--N\"orlund's integral. We
discuss the fundamental concepts of a variational calculus, such as
the Euler--Lagrange equations for the basic and isoperimetric
problems, as well as Lagrange and optimal control problems. As
particular cases we obtain the classical discrete-time calculus of
variations \cite[Chap.~8]{Kelley:Peterson}, the variational
$q$-calculus \cite{Bang04,Bangerezako},
and the calculus of variations applied
to N\"orlund's sum \cite{Fort1937}.

The plan of the paper is as follows. In Sect.~\ref{sec:Prel} some
basic formulas of Hahn's difference operator and the associated
Jackson--N\"orlund integral calculus are briefly reviewed. Our
results are formulated and proved in Sect.~\ref{sec:mainResults}.
Main results of the paper include necessary optimality conditions
for the basic problem of calculus of variations
(Theorem~\ref{thm:mr} in Sect.~\ref{ssec:EL}) and the
isoperimetric problem (Theorems~\ref{th:iso} and \ref{th:iso:abn} in
Sect.~\ref{ssec:ISO}), as well as a sufficient optimality condition
for the basic problem (Theorem~\ref{suff} in Sect.~\ref{ssuff}).
In Sect.~\ref{Leitmann} we show that the direct method introduced
by Leitmann in the sixties of the XX century \cite{Leitmann67}, can
also be applied with success to quantum variational problems via
Hahn's difference operator and Jackson--N\"orlund's integral.
Leitmann's method is a venerable forty years old direct method that
has shown through the times to be an universal and useful method in
several different contexts
--- see, \textrm{e.g.}, \cite{Car,CarlsonLeitmann05a,CarlsonLeitmann05b,%
TE,Leit,Leitmann01,MR1954118,MR2065731,MR2035262,mal:tor,SilvaTorres06,withLeitmann,Wagener}.
Sect.~\ref{exam} provides concrete examples of application of our
results. Finally, in Sect.~\ref{app} we apply the developed Hahn
variational calculus to obtain a quantum version of the well known
Ramsey model and we finish the paper with conclusions in Sect.~\ref{sec:conc}.

% ---------------------------------------------------

\section{Preliminaries}
\label{sec:Prel}

Let $q \in ]0,1[$ and $\omega \in ]0,+\infty[$ be given.
Define $\omega_0 := \frac{\omega}{1-q}$.
Throughout all the paper we assume $I$ to be an interval of
$\mathbb{R}$ containing $\omega_0$.

\begin{defin}[Hahn's difference operator]
Let $f : I \rightarrow \mathbb{R}$.
The Hahn difference operator is defined by
\begin{equation*}
D_{q,\omega} f(t):=
\begin{cases}
\frac{f(qt+\omega)-f(t)}{(qt+\omega)-t}, & \text{ if } t\neq \omega_0\\
f'(t), & \text{ if } t = \omega_0
\end{cases}
\end{equation*}
provided that $f$ is differentiable at $\omega_0$
(we are using $f'(t)$ to denote the Fr\'{e}chet derivative).
In this case, we call $D_{q,\omega}f$
the $q,\omega$-derivative of $f$ and say that $f$ is
$q,\omega$-differentiable on $I$.
\end{defin}

\begin{ex}
Let $q\in]0,1[$, $\omega=0$, and
\begin{equation*}
f(t)=
\begin{cases}
t^2, & \text{ if } t\in\mathbb{Q}\\
-t^2, & \text{ if } t\in\mathbb{R}\setminus\mathbb{Q}.
\end{cases}
\end{equation*}
In this case $\omega_0=0$, and $f$ is $q,\omega$-differentiable on the entire real line.
However, $f$ is Fr\'{e}chet differentiable in zero only.
\end{ex}

\begin{ex}
\label{ex:disctF:withD}
Let $q=\omega=1/2$. In this case $\omega_0=1$.
It is easy to see that $f : [-1,1] \rightarrow \mathbb{R}$ given by
\begin{equation*}
f(t) =
\begin{cases}
-t, & \text{ if } t \in (-1,0)\cup (0,1]\\
0, & \text{ if } t=-1\\
1, & \text{ if } t=0
\end{cases}
\end{equation*}
is not a continuous function but
is $q,\omega$-differentiable in $[-1,1]$ with
\begin{equation*}
D_{q,\omega} f(t) =
\begin{cases}
-1, & \text{ if } t \in (-1,0)\cup (0,1]\\
1, & \text{ if } t=-1\\
-3, & \text{ if }  t=0.
\end{cases}
\end{equation*}
\end{ex}

Note that,
\begin{equation*}
\lim_{q \uparrow 1}D_{q,\omega}f(t)=\Delta_{\omega} f(t), \quad
\lim_{\omega\downarrow 0}D_{q,\omega}f(t)=D_{q}f(t), \quad
\text{and } \lim_{\omega\downarrow 0,q\uparrow 1}D_{q,\omega}f(t)=f'(t),
\end{equation*}
where
\begin{equation*}
D_{q}f(t):=\frac{f(qt)-f(t)}{qt-t}, \quad t \neq 0,
\end{equation*}
is the Jackson $q$-difference derivative \cite{Gasper,Jackson1,Kac},
and
\begin{equation*}
\Delta_{\omega} f(t):=\frac{f(t+\omega)-f(t)}{(t+\omega)-t}
\end{equation*}
is the forward difference \cite{Bird,Jagerman,Jordan}.

\begin{rem}
Let $\eta_{q,\omega}(t) := qt + \omega$, $t \in I$.
Since $q \in ]0,1[$ and $\omega \in ]0,+\infty[$, then
$\eta_{q,\omega}(t)$ is a contraction,
$\eta_{q,\omega}(I)\subseteq I$,
$\eta_{q,\omega}(t) < t$ for $t
> \omega_0$, $\eta_{q,\omega}(t)> t$ for
$t < \omega_0$, and $\eta_{q,\omega}(\omega_0) = \omega_0$.
\end{rem}

The Hahn difference operator has the following properties:

\begin{thm}[\cite{Aldwoah,Annaby}]
\label{derv:prop}
\begin{itemize}
\item[(a)] Let $f$ be $q,\omega$-differentiable on $I$ and
$D_{q,\omega}f\equiv 0$ on $I$. Then $f$ is a constant. Conversely,
$D_{q,\omega}c=0$ for any constant $c$.
\item[(b)]Let $f$, $g$ be $q,\omega$-differentiable at $t\in I$. Then,
\begin{itemize}
\item[(i)]
$D_{q,\omega}(f+g)(t)=D_{q,\omega}f(t)+D_{q,\omega}g(t),$
\item[(ii)]
$D_{q,\omega}(fg)(t)=D_{q,\omega}(f(t))g(t)+
f(qt+\omega)D_{q,\omega}g(t),$
\item[(iii)]
$D_{q,\omega}\left(\frac{f}{g}\right)(t)
=\frac{D_{q,\omega}(f(t))g(t)-f(t)D_{q,\omega}g(t)}{g(t)g(qt+\omega)}$
provided $g(t)g(qt+\omega)\neq 0$.
\end{itemize}
\item[(c)]
$f(qt+\omega)=f(t)+((qt+\omega)-t)D_{q,\omega}f(t)$, $t\in I$.
\end{itemize}
\end{thm}

\begin{ex}[\cite{Aldwoah,Annaby}]
Let $a,b\in\mathbb{R}$. We have
\begin{equation}\label{tpower}
D_{q,\omega}(at+b)^n=a\sum_{k=0}^{n-1}(a(qt+\omega)+b)^k(at+b)^{n-k-1},
\end{equation}
for $n\in\mathbb{N}$ and $t \neq \omega_0$.
\end{ex}

Following \cite{Aldwoah,Annaby,Hamza}, we define the inverse of the
operator $D_{q,\omega}$:

\begin{defin}
Let $I$ be a closed interval of $\mathbb{R}$ such that
$\omega_0, a, b\in I$. For $f:I\rightarrow
\mathbb{R}$ we define the $q,\omega$-integral of $f$ from $a$ to $b$
by
\begin{equation}
\label{int:1}
\int_{a}^{b}f(t)d_{q,\omega}t
:=\int_{\omega_0}^{b}f(t)d_{q,\omega}t-\int_{\omega_0}^{a}f(t)d_{q,\omega}t,
\end{equation}
where
\begin{equation}
\label{int:2}
\int_{\omega_0}^{x}f(t)d_{q,\omega}t
:=(x(1-q)-\omega)\sum_{k=0}^{\infty}q^kf(xq^k+[k]_{q,\omega}), \quad
x\in I,
\end{equation}
with $[k]_{q,\omega} :=\frac{\omega(1-q^k)}{1-q}$ for $k\in
\mathbb{N}_0 =\mathbb{N} \cup \{0\}$,
provided that the series converges at $x=a$ and $x=b$. In this case,
$f$ is called $q,\omega$-integrable on $[a,b]$. We say that $f$ is
$q,\omega$-integrable over $I$ iff it is $q,\omega$-integrable over
$[a,b]$, for all $a,b\in I$.
\end{defin}

Note that, in the integral formulas
\eqref{int:1} and \eqref{int:2}, when $\omega \downarrow 0$ we
obtain the Jackson $q$-integral
\begin{equation*}
\int_{a}^{b}f(t)d_{q}t=\int_{0}^{b}f(t)d_{q}t-\int_{0}^{a}f(t)d_{q}t,
\end{equation*}
where
\begin{equation*}
\int_{0}^{x}f(t)d_{q}t=x(1-q)\sum_{k=0}^{\infty}q^kf(xq^k)
\end{equation*}
(see, \textrm{e.g.}, \cite{Jackson2});
while if $q\uparrow 1$ we obtain  the N\"orlund sum
\begin{equation*}
\int_{a}^{b}f(t)\Delta_{\omega}t
=\int_{+\infty}^{b}f(t)\Delta_{\omega}t-\int_{+\infty}^{a}f(t)\Delta_{\omega}t,
\end{equation*}
where
\begin{equation*}
\int_{+\infty}^{x}f(t)\Delta_{\omega}t=-\omega\sum_{k=0}^{+\infty}f(x+k\omega)
\end{equation*}
(see, \textrm{e.g.}, \cite{Fort,Jagerman,nol}).
This is why the integral defined by \eqref{int:1} and \eqref{int:2}
is called the Jackson--N\"orlund integral.

If $f:I\rightarrow\mathbb{R}$ is continuous at $\omega_0$, then $f$
is $q,\omega$-integrable over $I$ (see \cite{Aldwoah,Annaby} for the proof).

\begin{thm}[\cite{Aldwoah,Annaby}]
Assume $f:I\rightarrow \mathbb{R}$ be continuous at $\omega_0$.
Define
\begin{equation*}
F(x):=\int_{\omega_0}^{x}f(t)d_{q,\omega}t.
\end{equation*}
Then $F$ is continuous at $\omega_0$. Furthermore,
$D_{q,\omega}F(x)$ exists for every $x\in I$ and
$D_{q,\omega}F(x)=f(x)$. Conversely,
\begin{equation*}
\int_{a}^{b}D_{q,\omega}f(t)d_{q,\omega}t= f(b)-f(a)
\end{equation*}
for all $a,b\in I$.
\end{thm}

The $q,\omega$-integral has the following properties.

\begin{thm}[\cite{Aldwoah,Annaby}]
\label{integr:prop}
\begin{itemize}
\item[(a)] Let $f,g:I\rightarrow \mathbb{R}$ be $q,\omega$-integrable on
$I$, $a$, $b$, $c\in I$ and $k\in \mathbb{R}$. Then,
\begin{itemize}
\item[(i)]
$\int_{a}^{a}f(t)d_{q,\omega}t=0$,
\item[(ii)]
$\int_{a}^{b}kf(t)d_{q,\omega}t=k\int_{a}^{b}f(t)d_{q,\omega}t,$
\item[(iii)]
$\int_{a}^{b}f(t)d_{q,\omega}t=-\int_{b}^{a}f(t)d_{q,\omega}t,$
\item[(iv)]
$\int_{a}^{b}f(t)d_{q,\omega}t
=\int_{a}^{c}f(t)d_{q,\omega}t+\int_{c}^{b}f(t)d_{q,\omega}t,$
\item[(v)]
$\int_{a}^{b}(f(t)+g(t))d_{q,\omega}t
=\int_{a}^{b}f(t)d_{q,\omega}t+\int_{a}^{b}g(t)d_{q,\omega}t$.
\end{itemize}
\item[(b)] Every Riemann integrable function $f$ on $I$ is
$q,\omega$-integrable on $I$.
\item[(c)] If $f,g:I\rightarrow \mathbb{R}$ are continuous at $\omega_0$,
then
\begin{equation*}
\left.\int_{a}^{b}f(t)D_{q,\omega}g(t)d_{q,\omega}t
=f(t)g(t)\right|_{t=a}^{t=b}-\int_{a}^{b}D_{q,\omega}(f(t))g(qt+\omega)d_{q,\omega}t,
\quad a,b\in I.
\end{equation*}
\end{itemize}
\end{thm}
Property (c) in Theorem~\ref{integr:prop} is the integration by
parts formula for the Jackson--N\"orlund integral, and will be
useful in the proof of our Theorem~\ref{thm:mr}.

\begin{lem}[\cite{Aldwoah,Annaby}]
\label{integr:ineq} Let $s\in I$, $f$ and $g$ be
$q,\omega$-integrable. If $|f(t)|\leq g(t)$ for all
$t\in\{q^ns+[n]_{q,\omega}: n\in\mathbb{N}_{0}\}$, $s\in I$, then
\begin{equation*}
\left|\int_{\omega_0}^{b}f(t)d_{q,\omega}t\right|\leq\int_{\omega_0}^{b}g(t)d_{q,\omega}t,
\quad
\left|\int_{a}^{b}f(t)d_{q,\omega}t\right|\leq\int_{a}^{b}g(t)d_{q,\omega}t
\end{equation*}
for all $a,b\in\{q^ns+[n]_{q,\omega}: n\in\mathbb{N}_{0}\}$.
\end{lem}
However, it should be noted that the inequality
\begin{equation}
\label{surp:ineq}
\left|\int_{a}^{b}f(t)d_{q,\omega}t\right|\leq\int_{a}^{b}|f(t)|d_{q,\omega}t,
\end{equation}
$a,b\in I$, is not always true. For an example we refer the
reader to \cite{Aldwoah,Annaby}.

% ---------------------------------------------------

\section{Main Results}
\label{sec:mainResults}

Let $a,b\in I$ with $a < b$. We define the $q,\omega$-interval by
$$[a,b]_{q,\omega}:=\{q^na+[n]_{q,\omega}:
n\in\mathbb{N}_{0}\}\cup\{q^nb+[n]_{q,\omega}:
n\in\mathbb{N}_{0}\}\cup\{\omega_0\}.$$ By $\mathcal{D}$ we denote
the set of all real valued functions defined on $[a,b]_{q,\omega}$
and continuous at $\omega_0$.

\begin{lem}[Fundamental Lemma of the Hahn quantum variational calculus]
\label{lemma:DR} Let $f\in \mathcal{D}$. One has
$\int_{a}^{b}f(t)h(qt+\omega)d_{q,\omega}t=0$ for all functions
$h\in \mathcal{D}$ with $h(a)=h(b)=0$ if and only if $f(t)=0$ for
all $t\in[a,b]_{q,\omega}$.
\end{lem}

\begin{proof}
The implication ``$\Leftarrow$'' is obvious. Let us prove the
implication ``$\Rightarrow$''. Suppose, by contradiction, that
$f(p)\neq 0$ for some $p\in[a,b]_{q,\omega}$. \\
\emph{Case I}. If $p\neq \omega_0$, then $p=q^ka+[k]_{q,\omega}$ or
$p=q^kb+[k]_{q,\omega}$ for some $k\in\mathbb{N}_{0}$. Observe that
$a(1-q)-\omega$ and $b(1-q)-\omega$ cannot vanish simultaneously.
Therefore, without loss of generality, we can assume
$a(1-q)-\omega\neq0$ and $p=q^ka+[k]_{q,\omega}$. Define
$$
h(t)=
\begin{cases}
f(q^ka+[k]_{q,\omega}),  & \text{ if } t=q^{k+1}a+[k+1]_{q,\omega}\\
0, & \mbox{ otherwise}\, .
\end{cases}
$$
Then,
$$
\int_{a}^{b}f(t)h(qt+\omega)d_{q,\omega}t= -(a(1-q)-\omega)
q^kf(q^ka+[k]_{q,\omega})h(q^{k+1}a+[k+1]_{q,\omega})\neq 0,
$$
which is a contradiction.\\
\emph{Case II}. If $p=\omega_0$, then without loss of generality we
can assume $f(\omega_0)>0$. We know that (see \cite{Aldwoah,Annaby}
for more details)
$$
\lim_{n\uparrow\infty}q^na+[n]_{q,\omega}
=\lim_{n\uparrow\infty}q^nb+[n]_{q,\omega}=\omega_0.
$$
As $f$ is continuous at $\omega_0$, we have
$$
\lim_{n\uparrow\infty}f(q^na+[n]_{q,\omega})
=\lim_{n\uparrow\infty}f(q^nb+[n]_{q,\omega})=f(\omega_0).
$$
Therefore, there exists $N\in\mathbb{N}$, such that for all $n>N$ the
inequalities
$$
f(q^na+[n]_{q,\omega})>0 \ \ \text{ and } \ \  f(q^nb+[n]_{q,\omega})>0
$$
hold. If $\omega_0\neq a,b$, then we define
$$
h(t)=
\begin{cases}
f(q^nb+[n]_{q,\omega}),  & \text{ if } t=q^{n+1}a+[n+1]_{q,\omega},
\quad \mbox{for all}\quad  n>N\\
f(q^na+[n]_{q,\omega}),  & \text{ if } t=q^{n+1}b+[n+1]_{q,\omega},
\quad \mbox{for all}\quad  n>N\\
0, & \mbox{ otherwise}\, .
\end{cases}
$$
Hence,
$$
\int_{a}^{b}f(t)h(qt+\omega)d_{q,\omega}t=
(b-a)(1-q)\sum_{n=N}^{\infty}
q^nf(q^na+[n]_{q,\omega})f(q^{n}b+[n]_{q,\omega})\neq 0,
$$
which is a contradiction. If $\omega_0= b$, then we define
$$
h(t)=
\begin{cases}
f(\omega_0),  & \text{ if } t=q^{n+1}a+[n+1]_{q,\omega},
\quad \mbox{for all}\quad  n>N\\
0, & \mbox{ otherwise}\, .
\end{cases}
$$
Hence,
$$
\int_{a}^{b}f(t)h(qt+\omega)d_{q,\omega}t
=-\int_{\omega_0}^{a}f(t)h(qt+\omega)d_{q,\omega}t
=-(a(1-q)-\omega)\sum_{n=N}^{\infty}
q^nf(q^na+[n]_{q,\omega})f(\omega_0)\neq 0,
$$
which is a contradiction. Similarly,
we show the case when $\omega_0=a$.
\end{proof}

Let $\mathbb{E}$ be the linear space of functions $y\in \mathcal{D}$
for which the $q,\omega$-derivative is bounded on $[a,b]_{q,\omega}$
and continuous at $\omega_0$. We equip $\mathbb{E}$ with the norm
\begin{equation*}
\|y\|_{1}=\sup_{t\in [a,b]_{q,\omega}}|y(t)| +\sup_{t\in
[a,b]_{q,\omega}}|D_{q,\omega}y(t)|.
\end{equation*}
For $s\in I$ we set
$$
[s]_{q,\omega}=\{q^ns+[n]_{q,\omega}:
n\in\mathbb{N}_{0}\}\cup\{\omega_0\}.
$$

In the sequel we need one more result. The following definition and
lemma are similar to \cite{B:04}.

\begin{defin}
Let
$g:[s]_{q,\omega}\times]-\bar{\theta},\bar{\theta}[\rightarrow\mathbb{R}$.
We say that $g(t,\cdot)$ is continuous in $\theta_{0}$, uniformly in
$t$, iff for every $\varepsilon>0$ there exists $\delta>0$ such that
$|\theta-\theta_{0}|<\delta$ implies
$|g(t,\theta)-g(t,\theta_{0})|<\varepsilon$ for all
$t\in[s]_{q,\omega}$. Furthermore, we say that $g(t,\cdot)$ is
differentiable at $\theta_{0}$, uniformly in $t$, iff for every
$\varepsilon>0$ there exists $\delta>0$ such that
$0<|\theta-\theta_{0}|<\delta$ implies
$$
\left|\frac{g(t,\theta)-g(t,\theta_{0})}{\theta
-\theta_{0}}-\partial_2g(t,\theta_{0})\right|<\varepsilon,
$$
where $\partial_2g=\frac{\partial g}{\partial \theta}$ for all $t\in[s]_{q,\omega}$.
\end{defin}

\begin{lem}
\label{fun}
Assume $g(t,\cdot)$ be differentiable at $\theta_0$, uniformly in $t$
in $[s]_{q,\omega}$, and that
$G(\theta):=\int_{\omega_0}^sg(t,\theta)d_{q,\omega}t$, for $\theta$
near $\theta_0$, and
$\int_{\omega_0}^s\partial_2g(t,\theta_0)d_{q,\omega}$ exist. Then,
$G(\theta)$ is differentiable at $\theta_0$ with
$G'(\theta_0)=\int_{\omega_0}^s\partial_2g(t,\theta_0)d_{q,\omega}t$.
\end{lem}

\begin{proof}
Let $\varepsilon>0$ be arbitrary. Since $g(t,\cdot)$ is
differentiable at $\theta_{0}$, uniformly in $t$, there exists
$\delta>0$, such that, for all $t\in[s]_{q,\omega}$, and for
$0<|\theta-\theta_{0}|<\delta$, the following inequality holds:
$$
\left|\frac{g(t,\theta)-g(t,\theta_{0})}{\theta-\theta_{0}}
-\partial_2g(t,\theta_{0})\right|<\frac{\varepsilon}{s-\omega_0}.
$$
Applying Theorem~\ref{integr:prop} and Lemma~\ref{integr:ineq}, for
$0<|\theta-\theta_{0}|<\delta$, we have
\begin{multline*}
\left|\frac{G(\theta)-G(\theta_{0})}{\theta-\theta_{0}}-G'(\theta_{0})\right|
=\left|\frac{\int_{\omega_0}^sg(t,\theta)d_{q,\omega}t-\int_{\omega_0}^s
g(t,\theta_{0})d_{q,\omega}t}{\theta-\theta_{0}}
-\int_{\omega_0}^s\partial_2g(t,\theta_{0})d_{q,\omega}t\right|\\
=\left|\int_{\omega_0}^s\left[\frac{g(t,\theta)-g(t,\theta_{0})}{\theta-\theta_{0}}
-\partial_2g(t,\theta_{0})\right]d_{q,\omega}t\right|
<\int_{\omega_0}^s\frac{\varepsilon}{s-\omega_0}d_{q,\omega}t
=\frac{\varepsilon}{s-\omega_0}\int_{\omega_0}^s1d_{q,\omega}t
=\varepsilon.
\end{multline*}
Hence, $G(\cdot)$ is differentiable at $\theta_0$ and
$G'(\theta_0)=\int_{\omega_0}^s\partial_2g(t,\theta_0)d_{q,\omega}t$.
\end{proof}

%----------------------------------------------------

\subsection{The Hahn Quantum Euler--Lagrange Equation}
\label{ssec:EL}

We consider the variational problem of finding minima (or maxima) of
a functional
\begin{equation}
\label{vp} \mathcal{L}[y]
=\int_{a}^{b}f(t,y(qt+\omega),D_{q,\omega}y(t))d_{q,\omega} t
\end{equation}
over all $y\in\mathbb{E}$ satisfying the boundary conditions
\begin{equation}
\label{bc} y(a)=\alpha,\, \quad y(b)=\beta, \quad \alpha,\beta\in
\mathbb{R},
\end{equation}
where $f:[a,b]_{q,\omega}\times\mathbb{R}\times\mathbb{R}\rightarrow
\mathbb{R}$ is a given function. A function $y\in \mathbb{E}$ is
said to be admissible iff it satisfies the endpoint conditions
\eqref{bc}. Let us denote by $\partial_2f$ and $\partial_3f$,
respectively, the partial derivatives of $f(\cdot,\cdot,\cdot)$ with
respect to its second and third argument. In the sequel, we assume that
$(u,v)\rightarrow f(t,u,v)$ be a $C^1(\mathbb{R}^{2}, \mathbb{R})$
function for any $t \in [a,b]_{q,\omega}$, and
$f(\cdot,y(\cdot),D_{q,\omega}y(\cdot))$,
$\partial_2f(\cdot,y(\cdot),D_{q,\omega}y(\cdot))$, and
$\partial_3f(\cdot,y(\cdot),D_{q,\omega}y(\cdot))$ are continuous at
$\omega_0$ for any admissible function $y(\cdot)$. Finally, an
$h\in\mathbb{E}$ is called an admissible variation provided
$h(a)=h(b)=0$.

For an admissible variation $h$, we define function
$\phi : \, ]-\bar{\varepsilon},\bar{\varepsilon}[\rightarrow \mathbb{R}$
by
$$
\phi(\varepsilon) = \phi(\varepsilon;y,h)
:=\mathcal{L}[y + \varepsilon h].
$$
The first variation of problem
\eqref{vp}--\eqref{bc} is defined by
$$\delta\mathcal{L}[y,h]:=\phi(0;y,h)=\phi'(0).$$
Observe that,
\begin{equation*}
\begin{split}
\mathcal{L}[y + \varepsilon h]
&= \int_a^b f(t,y(qt+\omega)+\varepsilon h(qt+\omega),D_{q,\omega}y(t)
+ \varepsilon D_{q,\omega}h(t)) d_{q,\omega} t\\
&=\int_{\omega_0}^b f(t,y(qt+\omega)+\varepsilon
h(qt+\omega),D_{q,\omega}y(t)+ \varepsilon D_{q,\omega}h(t))
d_{q,\omega} t\\
&\qquad -\int_{\omega_0}^a f(t,y(qt+\omega)+\varepsilon
h(qt+\omega),D_{q,\omega}y(t)+ \varepsilon D_{q,\omega}h(t))
d_{q,\omega} t.
\end{split}
\end{equation*}
Writing
\begin{equation*}
\mathcal{L}_b[y + \varepsilon h]=\int_{\omega_0}^b
f(t,y(qt+\omega)+\varepsilon h(qt+\omega),D_{q,\omega}y(t)+
\varepsilon D_{q,\omega}h(t)) d_{q,\omega} t
\end{equation*}
and
\begin{equation*}
\mathcal{L}_a[y + \varepsilon h]=\int_{\omega_0}^a
f(t,y(qt+\omega)+\varepsilon h(qt+\omega),D_{q,\omega}y(t)+
\varepsilon D_{q,\omega}h(t)) d_{q,\omega} t,
\end{equation*}
we have $$\mathcal{L}[y + \varepsilon h]=\mathcal{L}_b[y +
\varepsilon h]-\mathcal{L}_a[y + \varepsilon h].$$ Therefore,
\begin{equation}\label{var}
\delta\mathcal{L}[y,h]=\delta\mathcal{L}_b[y,h]-\delta\mathcal{L}_a[y,
h].
\end{equation}
Knowing \eqref{var}, the following lemma is a direct consequence of
Lemma~\ref{fun}.

\begin{lem}
\label{asump}
Put $g(t,\varepsilon)=f(t,y(qt+\omega)+\varepsilon
h(qt+\omega),D_{q,\omega}y(t)+ \varepsilon D_{q,\omega}h(t))$ for
$\varepsilon \in ]-\bar{\varepsilon},\bar{\varepsilon}[$. Assume
that:
\begin{itemize}
\item[(i)] $g(t,\cdot)$ be differentiable at $0$, uniformly in
$t \in [a]_{q,\omega}$, and $g(t,\cdot)$ be differentiable at $0$,
uniformly in $t \in [b]_{q,\omega}$;
\item[(ii)] $\mathcal{L}_a[y + \varepsilon h]$ and $\mathcal{L}_b[y + \varepsilon
h]$, for $\varepsilon$ near $0$, exist;
\item[(iii)] $\int_{\omega_0}^a\partial_2g(t,0)d_{q,\omega}t$ and
$\int_{\omega_0}^b\partial_2g(t,0)d_{q,\omega}t$ exist.
\end{itemize}
Then,
$$
\delta\mathcal{L}[y,h]=\int_a^b \left[ \partial_2
f(t,y(qt+\omega),D_{q,\omega}y(t)) h(qt+\omega)
+ \partial_3 f(t,y(qt+\omega),D_{q,\omega}y(t))
D_{q,\omega}h(t)\right]d_{q,\omega}t.
$$
\end{lem}

In the sequel, we always assume, without mentioning it explicitly,
that variational problems satisfy the assumptions of Lemma~\ref{asump}.

\begin{defin}
An admissible function $\tilde{y}$ is said to be a local minimizer
(resp. a local maximizer) to problem \eqref{vp}--\eqref{bc} iff there
exists $\delta >0$, such that $\mathcal{L}[\tilde{y}] \leq
\mathcal{L}[y]$ (resp. $\mathcal{L}[\tilde{y}] \geq \mathcal{L}[y]$)
for all admissible $y$ with $\|y-\tilde{y}\|_{1}<\delta$.
\end{defin}

The following result offers a necessary condition for local
extremizer.

\begin{thm}[A necessary optimality condition for problem \eqref{vp}--\eqref{bc}]
\label{nec:con}
Suppose that the optimal path to problem \eqref{vp}--\eqref{bc} exists
and is given by $\tilde{y}$. Then, $\delta\mathcal{L}[\tilde{y},h]=0$.
\end{thm}

\begin{proof}
Without loss of generality, we can assume $\tilde{y}$ to be a local
minimizer. Let $h$ be any admissible variation and define a function
$\phi : \, ]-\bar{\varepsilon},\bar{\varepsilon}[\rightarrow \mathbb{R}$
by $\phi(\varepsilon) = \mathcal{L}[\tilde{y} + \varepsilon h]$.
Since $\tilde{y}$ is a local minimizer, there exists $\delta >0$, such
that $\mathcal{L}[\tilde{y}] \leq \mathcal{L}[y]$ for all admissible
$y$ with $\|y-\tilde{y}\|_{1}<\delta$. Therefore, $\phi(\varepsilon)
= \mathcal{L}[\tilde{y} + \varepsilon
h]\geq\mathcal{L}[\tilde{y}]=\phi(0)$ for all
$\varepsilon<\frac{\delta}{\|h\|_{1}}$. Hence, $\phi$ has a local
minimum at $\varepsilon=0$, and thus our assertion follows.
\end{proof}

\begin{thm}[The Hahn quantum Euler--Lagrange equation for problem \eqref{vp}--\eqref{bc}]
\label{thm:mr} Suppose that the optimal path to problem
\eqref{vp}--\eqref{bc} exists and is given by $\tilde{y}$. Then,
\begin{equation}
\label{Euler}
D_{q,\omega}\partial_3f(t,\tilde{y}(qt+\omega),D_{q,\omega}\tilde{y}(t))
=\partial_2f(t,\tilde{y}(qt+\omega),D_{q,\omega}\tilde{y}(t))
\end{equation}
for all $t \in [a,b]_{q,\omega}$.
\end{thm}

\begin{proof}
Suppose that $\mathcal{L}$ has a local extremum at $\tilde{y}$. Let
$h$ be any admissible variation and define a function
$\phi : \, ]-\bar{\varepsilon},\bar{\varepsilon}[\rightarrow \mathbb{R}$
by $\phi(\varepsilon) = \mathcal{L}[\tilde{y} + \varepsilon h]$. By
Theorem~\ref{nec:con}, a necessary condition for $\tilde{y}$ to be
an extremizer is given by
\begin{equation}
\label{eq:FT} \left.\phi'(\varepsilon)\right|_{\varepsilon=0} = 0
\Leftrightarrow \int_a^b \left[ \partial_2f(\cdots) h(qt+\omega) +
\partial_3f(\cdots) D_{q,\omega}h(t)\right]d_{q,\omega}t = 0 \, ,
\end{equation}
where $(\cdots) =
\left(t,\tilde{y}(qt+\omega),D_{q,\omega}\tilde{y}(t)\right)$.
Integration by parts (see item (c) in Theorem~\ref{integr:prop})
gives
\begin{equation*}
\int_a^b \partial_3f(\cdots) D_{q,\omega}h(t)d_{q,\omega}t
=\left.\partial_3f(\cdots)h(t)\right|_{t=a}^{t=b}-\int_a^b
 D_{q,\omega}\partial_3f(\cdots)h(qt+\omega)d_{q,\omega}t.
\end{equation*}
Because $h(a) = h(b)= 0$, the necessary condition \eqref{eq:FT} can
be written as
\begin{equation*}
0 = \int_a^b \left(\partial_2f(\cdots)
-D_{q,\omega}\partial_3f(\cdots)\right) h(qt+\omega)d_{q,\omega}t
\end{equation*}
for all $h$ such that $h(a) = h(b) = 0$. Thus, by
Lemma~\ref{lemma:DR}, we have
\begin{equation*}
\partial_2f(\cdots)-D_{q,\omega}\partial_3f(\cdots)=0
\end{equation*}
for all $t\in[a,b]_{q,\omega}$.
\end{proof}

\begin{rem}
If the function $f$ under the sign of integration (the Lagrangian)
is given by $f = f(t,y_1,\ldots, y_n,
D_{q,\omega}y_1,\ldots,D_{q,\omega}y_n)$, then the necessary
optimality condition is given by $n$ equations similar to
\eqref{Euler}, one equation for each variable.
\end{rem}

% ---------------------------------------------------

\subsection{The Hahn Quantum Isoperimetric Problem}
\label{ssec:ISO}

Let us consider now the isoperimetric problem,
which consists of minimizing or maximizing
\begin{equation}
\label{ivp} \mathcal{L}[y]
=\int_{a}^{b}f(t,y(qt+\omega),D_{q,\omega}y(t))d_{q,\omega} t
\end{equation}
over all $y\in \mathbb{E}$ satisfying the boundary conditions
\begin{equation}
\label{bivp}
 y(a)=\alpha,\, \quad y(b)=\beta,
\end{equation}
and the constraint
\begin{equation}
\label{civp} \mathcal{K}[y]
=\int_{a}^{b}g(t,y(qt+\omega),D_{q,\omega}y(t))d_{q,\omega}t =k,
\end{equation}
where $\alpha$, $\beta$, and $k$ are given real numbers. We assume
that:
\begin{itemize}
\item[(i)] $(u,v)\rightarrow f(t,u,v)$ and $(u,v)\rightarrow g(t,u,v)$
be $C^1(\mathbb{R}^{2}, \mathbb{R})$ functions for any $t \in
[a,b]_{q,\omega}$;
\item[(ii)] functions
$f(\cdot,y(\cdot)D_{q,\omega}y(\cdot))$,
$\partial_2f(\cdot,y(\cdot),D_{q,\omega}y(\cdot))$ and
$\partial_3f(\cdot,y(\cdot),D_{q,\omega}y(\cdot))$,
$g(\cdot,y(\cdot),D_{q,\omega}y(\cdot))$,
$\partial_2g(\cdot,y(\cdot),D_{q,\omega}y(\cdot))$ and
$\partial_3g(\cdot,y(\cdot),D_{q,\omega}y(\cdot))$, be continuous
at $\omega_0$ for any admissible function $y(\cdot)$.
\end{itemize}

\begin{defin}
An admissible function $\tilde{y}$ is said to be a local minimizer
(resp. local maximizer) for the isoperimetric problem
\eqref{ivp}--\eqref{civp} iff there exists $\delta >0$ such that
$\mathcal{L}[\tilde{y}]\leq \mathcal{L}[y]$ (resp.
$\mathcal{L}[\tilde{y}] \geq \mathcal{L}[y]$) for all admissible $y$
satisfying the boundary conditions \eqref{bivp}, the isoperimetric
constraint \eqref{civp}, and $\|y-\tilde{y}\|_{1}<\delta$.
\end{defin}

\begin{defin}
We say that $\tilde{y}$ is an extremal for $\mathcal{K}$, iff
\eqref{Euler} holds with respect to \eqref{civp}:
\begin{equation*}
D_{q,\omega}\partial_3
g(t,\tilde{y}(qt+\omega),D_{q,\omega}\tilde{y}(t))
=\partial_2g(t,\tilde{y}(qt+\omega),D_{q,\omega}\tilde{y}(t))
\end{equation*}
for all $t \in [a,b]_{q,\omega}$. An extremizer (\textrm{i.e.}, a
local minimizer or a local maximizer) for the problem
\eqref{ivp}--\eqref{civp}, that is not an extremal for $\mathcal{K}$,
is said to be a normal extremizer; otherwise (\textrm{i.e.}, if it
is an extremal for $\mathcal{K}$), the extremizer is said to be
abnormal.
\end{defin}

\begin{thm}
\label{th:iso} If $\tilde{y}$ is a normal extremizer for the
isoperimetric problem \eqref{ivp}--\eqref{civp}, then there exists a
real $\lambda$, such that
\begin{equation}
\label{iso}
D_{q,\omega}\partial_3F(t,\tilde{y}(qt+\omega),D_{q,\omega}\tilde{y}(t))
=\partial_2F(t,\tilde{y}(qt+\omega),D_{q,\omega}\tilde{y}(t))
\end{equation}
for all $t \in [a,b]_{q,\omega}$, where $F=f-\lambda g$.
\end{thm}

\begin{proof}
Consider a variation of $\tilde{y}$, say
$\bar{y}=\tilde{y}+\varepsilon_{1}h_{1}+\varepsilon_{2}h_{2}$, where
$h_{i}\in \mathbb{E}$, $h_{i}(a)=h_{i}(b)=0$, and $\varepsilon_{i}$
is a sufficiently small parameter, $i=1,2$. Here, $h_{1}$ is an
arbitrary fixed function and $h_{2}$ is a fixed function that will
be chosen later. Define
\begin{equation*}
\bar{K}(\varepsilon_{1},\varepsilon_{2})=\mathcal{K}[\bar{y}]
=\int_{a}^{b}g(t,\bar{y}(qt+\omega),D_{q,\omega}\bar{y}(t))d_{q,\omega}t-k.
\end{equation*}
We have
\begin{equation*}
\left.\frac{\partial\bar{K}}{\partial
\varepsilon_{2}}\right|_{(0,0)} =\int_a^b \left[ \partial_2g(\cdots)
h_2(qt+\omega) + \partial_3g(\cdots)
D_{q,\omega}h_2(t)\right]d_{q,\omega}t
\end{equation*}
where $(\cdots) =
\left(t,\tilde{y}(qt+\omega),D_{q,\omega}\tilde{y}(t)\right)$.
Integration by parts gives
\begin{equation*}
\left.\frac{\partial\bar{K}}{\partial
\varepsilon_{2}}\right|_{(0,0)} = \int_a^b
\left(\partial_2g(\cdots)-D_{q,\omega}\partial_3g(\cdots)\right)
h_2(qt+\omega)d_{q,\omega}t
\end{equation*}
since $h_{2}(a)=h_{2}(b)=0$. By Lemma~\ref{lemma:DR}, there exists
$h_{2}$ such that $\left.\frac{\partial\bar{K}}{\partial
\varepsilon_{2}}\right|_{(0,0)}\neq 0$. Since $\bar{K}(0,0)=0$,
by the Dini-Ljusternik implicit function theorem we conclude that there exists a
function $\varepsilon_{2}$, defined in the neighborhood of zero, such
that $\bar{K}(\varepsilon_{1},\varepsilon_{2}(\varepsilon_{1}))=0$,
\textrm{i.e.}, we may choose a subset of variations $\bar{y}$
satisfying the
isoperimetric constraint. \\
Let us now consider the real function
\begin{equation*}
\bar{L}(\varepsilon_{1},\varepsilon_{2})=\mathcal{L}[\bar{y}]=\int_a^b
f(t,\bar{y}(qt+\omega),D_{q,\omega}\bar{y}(t)) d_{q,\omega} t.
\end{equation*}
By hypothesis, $(0,0)$ is an extremal of $\bar{L}$ subject to the
constraint $\bar{K}=0$ and $\nabla \bar{K}(0,0)\neq \textbf{0}$. By
the Lagrange multiplier rule, there exists some real $\lambda$, such
that $\nabla(\bar{L}(0,0)-\lambda\bar{K}(0,0))=\textbf{0}$. Having
in mind that $h_{1}(a)=h_{1}(b)=0$, we can write
\begin{equation*}
\left.\frac{\partial\bar{L}}{\partial
\varepsilon_{1}}\right|_{(0,0)} =\int_a^b
\left(\partial_2f(\cdots)-D_{q,\omega}\partial_3f(\cdots)\right)
h_1(qt+\omega)d_{q,\omega}t
\end{equation*}
and
\begin{equation*}
\left.\frac{\partial\bar{K}}{\partial
\varepsilon_{1}}\right|_{(0,0)} =\int_a^b
\left(\partial_2g(\cdots)-D_{q,\omega}\partial_3g(\cdots)\right)
h_1(qt+\omega)d_{q,\omega}t.
\end{equation*}
Therefore,
\begin{equation}
\label{iso:3} \int_a^b
\left[\left(\partial_2f(\cdots)-D_{q,\omega}\partial_3f(\cdots)\right)
-\lambda\left(\partial_2g(\cdots)-D_{q,\omega}\partial_3g(\cdots)\right)\right]
h_1(qt+\omega)d_{q,\omega}t=0.
\end{equation}
As \eqref{iso:3} holds for any $h_{1}$, by Lemma~\ref{lemma:DR} we
have
\begin{equation*}
\partial_2f(\cdots)-D_{q,\omega}\partial_3f(\cdots)
-\lambda\left(\partial_2g(\cdots)-D_{q,\omega}\partial_3g(\cdots)\right)=0.
\end{equation*}
We get \eqref{iso} by writing $F=f-\lambda g$.
\end{proof}

One can easily cover abnormal extremizers within our result by
introducing an extra multiplier $\lambda_{0}$ associated with the
cost functional.

\begin{thm}
\label{th:iso:abn} If $\tilde{y}$ is an extremizer for the
isoperimetric problem \eqref{ivp}--\eqref{civp}, then there exist
two constants $\lambda_{0}$ and $\lambda$, not both zero, such that
\begin{equation}
\label{iso:abn}
D_{q,\omega}\partial_3F(t,\tilde{y}(qt+\omega),D_{q,\omega}\tilde{y}(t))
=\partial_2F(t,\tilde{y}(qt+\omega),D_{q,\omega}\tilde{y}(t))
\end{equation}
for all $t \in [a,b]_{q,\omega}$, where $F=\lambda_0f-\lambda g$.
\end{thm}

\begin{proof}
Following the proof of Theorem~\ref{th:iso}, since $(0,0)$ is an
extremal of $\bar{L}$ subject to the constraint $\bar{K}=0$, the
extended Lagrange multiplier rule (see, for instance,
\cite[Theorem~4.1.3]{Brunt}) asserts the existence of reals
$\lambda_{0}$ and $\lambda$, not both zero, such that
$\nabla(\lambda_{0}\bar{L}(0,0)-\lambda\bar{K}(0,0))=\textbf{0}$.
Therefore,
\begin{equation*}
\lambda_{0}\left.\frac{\partial\bar{L}}{\partial
\varepsilon_{1}}\right|_{(0,0)}
-\lambda\left.\frac{\partial\bar{K}}{\partial
\varepsilon_{1}}\right|_{(0,0)}=0
\end{equation*}
\begin{equation}
\label{iso:3:abn} \Leftrightarrow \int_a^b
\left[\lambda_0\left(\partial_2f(\cdots)-D_{q,\omega}\partial_3f(\cdots)\right)
-\lambda\left(\partial_2g(\cdots)-D_{q,\omega}\partial_3g(\cdots)\right)\right]
h_1(qt+\omega)d_{q,\omega}t=0.
\end{equation}
Since \eqref{iso:3:abn} holds for any $h_{1}$, it follows by
Lemma~\ref{lemma:DR} that
\begin{equation*}
\lambda_0\left(\partial_2f(\cdots)-D_{q,\omega}\partial_3f(\cdots)\right)
-\lambda\left(\partial_2g(\cdots)-D_{q,\omega}\partial_3g(\cdots)\right)=0.
\end{equation*}
The desired condition \eqref{iso:abn} follows by taking
$F=\lambda_0f-\lambda g$.
\end{proof}

\begin{rem}
If $\tilde{y}$ is a normal extremizer for the isoperimetric problem
\eqref{ivp}--\eqref{civp}, then we can choose $\lambda_{0}=1$ in
Theorem~\ref{th:iso:abn} and obtain Theorem~\ref{th:iso}. For
abnormal extremizers, Theorem~\ref{th:iso:abn} holds with
$\lambda_{0}=0$. The condition $(\lambda_{0},\lambda)\neq\textbf{0}$
guarantees that Theorem~\ref{th:iso:abn} is a nontrivial necessary
condition. In general we cannot guarantee, a priori, that $\lambda_{0}$ be different from zero.
The interested reader about abnormality is referred to the book \cite{Arutyunov}.
\end{rem}

Suppose now that it is required to find functions $y_1$ and $y_2$
for which the functional
\begin{equation}
\label{funct2}
\mathcal{L}[y_1,y_2]=\int_{a}^{b}f(t,y_1(qt+\omega),y_2(qt+\omega),
D_{q,\omega}y_1(t),D_{q,\omega}y_2(t))d_{q,\omega} t
\end{equation}
has an extremum, where the admissible functions satisfy the boundary
conditions
\begin{equation}
\label{boundconst2} (y_1(a),y_2(a))=(y_1^a,y_2^a) \mbox{ and }
(y_1(b),y_2(b))=(y_1^b,y_2^b),
\end{equation}
and the subsidiary nonholonomic condition
\begin{equation}
\label{subsconst}
g(t,y_1(qt+\omega),y_2(qt+\omega),D_{q,\omega}y_1(t),D_{q,\omega}y_2(t))=0.
\end{equation}
The problem \eqref{funct2}--\eqref{subsconst} can be reduced to the
isoperimetric one by transforming \eqref{subsconst} into a
constraint of the type \eqref{civp}. For that, we multiply both
sides of \eqref{subsconst} by an arbitrary function $\lambda(t)$,
and then take the $q,\omega$-integral from $a$ to $b$. We obtain the
new constraint
\begin{equation}
\label{const}
\mathcal{K}[y_1,y_2]=\int_{a}^{b}\lambda(t)g(t,y_1(qt+\omega),
y_2(qt+\omega),D_{q,\omega}y_1(t),D_{q,\omega}y_2(t))d_{q,\omega}t=0.
\end{equation}
Under the conditions of Theorem~\ref{th:iso}, the solutions
$(y_1,y_2)$ of the isoperimetric problem \eqref{funct2} and
\eqref{const} satisfy the Euler--Lagrange equation for the
functional
\begin{equation}\label{new}
\int_{a}^{b}(f-\tilde{\lambda}(t)g)d_{q,\omega}t,
\end{equation}
$\tilde{\lambda}(t)=\bar{\lambda}\lambda(t)$ for some constant
$\bar{\lambda}$. Since \eqref{const} follows from \eqref{subsconst},
the solutions of problem \eqref{funct2}--\eqref{subsconst} satisfy
as well the Euler--Lagrange equation for functional \eqref{new}.

%----------------------------------------------------------

\subsection{Sufficient Conditions}
\label{ssuff}

In this section, we prove sufficient conditions that ensure the
existence of minimum (maximum). Similarly to what happens in the
classical calculus of variations, some hypotheses of convexity
(concavity) are in order.

\begin{defin}
Given a function $f$, we say that $f(\underline t,u,v)$ is jointly
convex (concave) in $(u,v)$, iff $\partial_i f$, $i=2,3$, exist and
are continuous and verify the following condition:
\begin{equation*}
f(t,u+u_1,v+v_1)-f(t,u,v) \geq (\leq) \partial_2
f(t,u,v)u_1+\partial_3 f(t,u,v)v_1
\end{equation*}
for all $(t,u,v)$,$(t,u+u_1,v+v_1)\in
[a,b]_{q,\omega}\times\mathbb{R}^2$.
\end{defin}

\begin{thm}
\label{suff} Let $f(\underline t,u,v)$ be jointly convex (concave)
in $(u,v)$. If $\tilde{y}$ satisfies condition \eqref{Euler}, then
$\tilde{y}$ is a global minimizer (maximizer) to problem
\eqref{vp}--\eqref{bc}.
\end{thm}

\begin{proof}
We give the proof for the convex case. Since $f$ is jointly convex
in $(u,v)$ for any admissible function $\tilde{y}+h$, we have
\begin{equation*}
\begin{split}
\mathcal{L}&(\tilde{y}+h)-\mathcal{L}(\tilde{y})\\
&=\int_a^b\left[f(t,\tilde{y}(qt+\omega)
+h(qt+\omega),D_{q,\omega}\tilde{y}(t)+D_{q,\omega}h(t))-f(t,\tilde{y}(qt+\omega),
D_{q,\omega}\tilde{y}(t))\right]d_{q,\omega}\\
&\geq\int_a^b\left[\partial_2f(t,\tilde{y}(qt+\omega),D_{q,\omega}\tilde{y}(t))h(qt+\omega)
+\partial_3f(t,\tilde{y}(qt+\omega),D_{q,\omega}\tilde{y}(t))D_{q,\omega}h(t)\right]d_{q,\omega}.
\end{split}
\end{equation*}
We can now proceed analogously to the proof of Theorem~\ref{thm:mr}.
As the result we get
\begin{multline*}
\mathcal{L}(\tilde{y}+h)-\mathcal{L}(\tilde{y})
\geq\left.\partial_3f(t,\tilde{y}(qt+\omega),
\tilde{y}(t))h(t)\right|_{t=a}^{t=b}\\
+\int_a^b \left(\partial_2f(t,\tilde{y}(qt+\omega),\tilde{y}(t))
-D_{q,\omega}\partial_3f(t,\tilde{y}(qt+\omega),\tilde{y}(t))\right)
h(qt+\omega)d_{q,\omega}t.
\end{multline*}
Since $\tilde{y}$ satisfy conditions \eqref{Euler} and
$h(a)=h(b)=0$, we obtain $\mathcal{L}(\tilde{y}+h) -
\mathcal{L}(\tilde{y}) \geq 0$.
\end{proof}

%----------------------------------------------------------

\subsection{Leitmann's Direct Method}
\label{Leitmann}

Leitmann's direct method permits to
compute global solutions to some problems
that are variationally invariant
under a family of transformations
\cite{Leitmann67,Leit,Leitmann01,SilvaTorres06,withLeitmann}.
It should be mentioned that such invariance transformations
are useful not only in connection
with Leitmann's method but also
to apply Noether's theorem \cite{Torres:EJC,Torres:CPAA}.
Moreover, the invariance transformations are
related with the notion of Carath\'{e}odory equivalence
\cite{Car,Torres:JMS}.

Recently, it has been noticed by the authors
that the invariance transformations,
that keep the Lagrangian invariant, do not depend
on the time scale \cite{mal:tor}.
This is also true for the generalized Hahn quantum
setting that we are considering in this work:
given a Lagrangian
$f : \mathbb{R} \times \mathbb{R} \times \mathbb{R} \rightarrow \mathbb{R}$,
the invariance transformations, that keep it invariant up to a gauge term,
are exactly the same if the Lagrangian $f$ is used
to define a Hahn quantum functional \eqref{vp} or a classical
functional $\mathcal{L}[y] =\int_{a}^{b} L(t,y(t),y'(t))dt$
of the calculus of variations. Thus, if the quantum problem we want to solve admits
an enough rich family of invariance transformations,
that keep it invariant up to a gauge term,
then one need not to solve a Hahn quantum Euler--Lagrange equation
to find its minimizer: instead, we can try to use Leitmann's direct method.
The question of how to find the invariance
transformations is addressed in
\cite{GouveiaTorres05,GouveiaTorresRocha}.

Let $\bar{f}:[a,b]_{q,\omega}\times \mathbb{R} \times \mathbb{R}
\rightarrow \mathbb{R}$. We assume $(u,v)\rightarrow \bar{f}(t,u,v)$
is a $C^1(\mathbb{R}^{2}, \mathbb{R})$ function for any $t \in
[a,b]_{q,\omega}$, and
$\bar{f}(\cdot,\bar{y}(\cdot),D_{q,\omega}\bar{y}(\cdot))$,
$\partial_2\bar{f}(\cdot,\bar{y}(\cdot),D_{q,\omega}\bar{y}(\cdot))$,
and
$\partial_3\bar{f}(\cdot,\bar{y}(\cdot),D_{q,\omega}\bar{y}(\cdot))$
are continuous in $\omega_0$ for any admissible function
$\bar{y}(\cdot)$. Consider the integral
\begin{equation*}
\bar{\mathcal{L}}[\bar{y}]
=\int_a^b\bar{f}(t,\bar{y}(qt+\omega),D_{q,\omega}\bar{y}(t))d_{q,\omega} t.
\end{equation*}

\begin{lem}[Leitmann's fundamental lemma via Hahn's quantum operator]
\label{Fund:lemma:Leit} Let $y=z(t,\bar{y})$ be a transformation
having an unique inverse $\bar{y}=\bar{z}(t,y)$ for all $t\in
[a,b]_{q,\omega}$, such that there is a one-to-one correspondence
\begin{equation*}
y(t)\Leftrightarrow \bar{y}(t)
\end{equation*}
for all functions $y\in \mathbb{E}$ satisfying \eqref{bc} and all
functions $\bar{y}\in \mathbb{E}$ satisfying
\begin{equation}
\label{bc:trans} \bar{y}=\bar{z}(a,\alpha), \quad
\bar{y}=\bar{z}(b,\beta).
\end{equation}
If the transformation $y=z(t,\bar{y})$ is such that there exists a
function $G:[a,b]_{q,\omega} \times \mathbb{R} \rightarrow
\mathbb{R}$ satisfying the functional identity
\begin{equation}
\label{id} f(t,y(qt+\omega),D_{q,\omega}y(t))
-\bar{f}(t,\bar{y}(qt+\omega),D_{q,\omega}\bar{y}(t))
=D_{q,\omega}G(t,\bar{y}(t)) \, ,
\end{equation}
then if  $\bar{y}^{*}$ yields the extremum of $\bar{\mathcal{L}}$
with $\bar{y}^{*}$ satisfying \eqref{bc:trans},
$y^{*}=z(t,\bar{y}^{*})$ yields the extremum of $\mathcal{L}$ for
$y^{*}$ satisfying \eqref{bc}.
\end{lem}

\begin{rem}
The functional identity \eqref{id} is exactly
the definition of variationally invariance
when we do not consider
transformations of the time variable $t$
(\textrm{cf.} (4) and (5) of \cite{withLeitmann}).
Function $G$ that appears in \eqref{id}
is sometimes called a gauge term \cite{Torres:CPAA}.
\end{rem}

\begin{proof}
The proof is similar in spirit to Leitmann's proof
\cite{Leitmann67,Leit,Leitmann01,MR2035262}. Let  $y\in \mathbb{E}$
satisfy \eqref{bc}, and define functions $\bar{y}\in \mathbb{E}$
through the formula $\bar{y}=\bar{z}(t,y)$, $t\in[a,b]_{q,\omega}$.
Then $\bar{y}\in \mathbb{E}$ and satisfies \eqref{bc:trans}.
Moreover, as a result of \eqref{id}, it follows that
\begin{equation*}
\begin{split}
\mathcal{L}[y]-\bar{\mathcal{L}}[\bar{y}]
&=\int_{a}^{b}f(t,y(qt+\omega),D_{q,\omega}y(t)) d_{q,\omega} t
-\int_{a}^{b}\bar{f}(t,\bar{y}(qt+\omega),D_{q,\omega}\bar{y}(t)) d_{q,\omega} t\\
&=\int_{a}^{b}D_{q,\omega}G(t,\bar{y}(t)) d_{q,\omega} t
=G(b,\bar{y}(b))-G(a,\bar{y}(a))\\
&=G(b,\bar{z}(b,\beta))-G(a,\bar{z}(a,\alpha)),
\end{split}
\end{equation*}
from which the desired conclusion follows immediately since the
right-hand side of the above equality is a constant, depending only
on the fixed-endpoint conditions \eqref{bc}.
\end{proof}

Examples~\ref{example2}, \ref{example4} and \ref{example3}
in the next section illustrate the applicability of Lemma~\ref{Fund:lemma:Leit}.
The procedure is as follows: (i)
we use the computer algebra package described in \cite{GouveiaTorres05}
and available from the \emph{Maple Application Center} at
\url{http://www.maplesoft.com/applications/view.aspx?SID=4805}
to find the transformations
that keep the problem of the calculus
of variations or optimal control invariant; (ii)
we use such invariance transformations
to solve the Hahn quantum variational problem by applying
Leitmann's fundamental lemma (Lemma~\ref{Fund:lemma:Leit}).

%----------------------------------------------------------

\subsection{Illustrative Examples}
\label{exam}

We provide some examples in order to illustrate our main results.

\begin{ex}
\label{example1} Let $q$, $\omega$ be fixed real numbers, and $I$ be
a closed interval of $\mathbb{R}$ such that $\omega_0,0,1\in I$.
Consider the problem
\begin{equation}
\label{AD} \text{minimize} \quad \mathcal{L}[y]
=\int_{0}^{1}\left(y(qt+\omega) +\frac{1}{2}
(D_{q,\omega}y(t))^2\right)d_{q,\omega} t
\end{equation}
subject to the boundary conditions
\begin{equation}
\label{eq:bc} y(0)=0,  \quad y(1)=1.
\end{equation}
If $y$ is a local minimizer to problem \eqref{AD}--\eqref{eq:bc},
then by Theorem~\ref{thm:mr} it satisfies the Euler--Lagrange
equation
\begin{equation}
\label{EL:ex} D_{q,\omega}D_{q,\omega}y(t)=1
\end{equation}
for all $t\in\{\omega[n]_q:
n\in\mathbb{N}_{0}\}\cup\{q^n+\omega[n]_q: n\in\mathbb{N}_{0}\} \cup
\{\omega_0\}$. By direct substitution it can be verified that
$y(t)=\frac{1}{q+1}t^2-\frac{1}{q+1}t$ is a candidate solution to
problem \eqref{AD}--\eqref{eq:bc}.
\end{ex}

In next examples we solve quantum variational problems using
Leitmann's direct method (see Sect.~\ref{Leitmann}).

\begin{ex}
\label{example2} Let $q$, $\omega$, and $a$, $b$ ($a < b$) be
fixed real numbers, and $I$ be a closed interval of $\mathbb{R}$
such that $\omega_0\in I$ and $a,b \in
\{q^ns+[n]_{q,\omega} : n\in\mathbb{N}_{0}\}\cup\{\omega_0\}$
for some $s \in I$.
Let $\alpha$ and $\beta$ be two given reals, $\alpha \ne \beta$.
We consider the following problem:
\begin{equation}
\label{illust:Ex:mod}
\begin{gathered}
\text{minimize} \quad \mathcal{L}[y] =\int_{a}^{b}
\left((D_{q,\omega}y(t))^2
+y(qt+\omega) +t D_{q,\omega}y(t)\right) d_{q,\omega} t \, , \\
y(a)=\alpha \, , \quad y(b)=\beta \, .
\end{gathered}
\end{equation}
We transform problem \eqref{illust:Ex:mod} into the trivial problem
\begin{equation*}
 \text{minimize} \quad
 \bar{\mathcal{L}}[\bar{y}]=\int_{a}^{b}(D_{q,\omega}\bar{y}(t))^2
d_{q,\omega} t \, , \quad
\bar{y}(a)=0 \, , \quad \bar{y}(b)=0 \, ,
\end{equation*}
which has solution $\bar{y}\equiv 0$. For that we consider the
transformation
\begin{equation*}
y(t)=\bar{y}(t)+ct+d, \quad c,d\in \mathbb{R},
\end{equation*}
where constants $c$ and $d$ will be chosen later. According to the
above, we have
\begin{equation*}
D_{q,\omega}y(t)=D_{q,\omega}\bar{y}(t)+c, \quad
y(qt+\omega)=\bar{y}(qt+\omega)+c(qt+\omega)+d,
\end{equation*}
and
\begin{equation*}
\begin{split}
(D_{q,\omega}y(t))^2
&+y(qt+\omega) +t D_{q,\omega}y(t)\\
&=(D_{q,\omega}\bar{y}(t))^2 +2cD_{q,\omega}\bar{y}(t)
+c^2+\bar{y}(qt+\omega)+c(qt+\omega)
+d + tD_{q,\omega}\bar{y}(t)+ct\\
&=(D_{q,\omega}\bar{y}(t))^2 +
D_{q,\omega}[2c\bar{y}(t)+t\bar{y}(t)+ct^2+(c^2+d)t].
\end{split}
\end{equation*}
In order to obtain the solution to the original problem, it suffices
to chose $c$ and $d$ so that
\begin{equation}
\label{eq:const:mod}
ca+d=\alpha\, , \quad
cb+d=\beta \, .
\end{equation}
Solving the system of equations \eqref{eq:const:mod} we obtain
$c=\frac{\alpha-\beta}{a-b}$ and $d =\frac{\beta a-b\alpha}{a-b}$.
Hence, the global minimizer for problem \eqref{illust:Ex:mod} is
\begin{equation*}
y(t)=\frac{\alpha-\beta}{a-b}t+\frac{\beta a-b\alpha}{a-b}.
\end{equation*}
\end{ex}

\begin{ex}
\label{example4}
Let $q$, $\omega$, and $a$, $b$ ($a < b$) be
fixed real numbers, and $I$ be a closed interval of $\mathbb{R}$
such that $\omega_0\in I$ and $a,b \in
\{q^ns+[n]_{q,\omega} : n\in\mathbb{N}_{0}\}\cup\{\omega_0\}$
for some $s \in I$. Let $\alpha$ and $\beta$ be two given
reals, $\alpha \ne \beta$. We consider the following problem:
\begin{equation}
\label{illust:Ex:4}
\text{minimize} \quad \mathcal{L}[y] =\int_{a}^{b}
\left[D_{q,\omega}(y(t)g(t))\right]^2 d_{q,\omega} t \, , \quad
y(a)=\alpha \, , \quad y(b)=\beta \, ,
\end{equation}
where $g$ does not vanish on the interval $[a,b]_{q,\omega}$.
Observe that $\bar{y}(t) = g^{-1}(t)$ minimizes $\mathcal{L}$ with
end conditions $\bar{y}(a) = g^{-1}(a)$ and $\bar{y}(b) =
g^{-1}(b)$. Let $y(t)=\bar{y}(t)+p(t)$. Then
\begin{equation}
\label{tran} \left[D_{q,\omega}(y(t)g(t))\right]^2
=\left[D_{q,\omega}(\bar{y}(t)g(t))\right]^2
+D_{q,\omega}(p(t)g(t))D_{q,\omega}\left(2\bar{y}(t)g(t)+p(t)g(t)\right).
\end{equation}
Consequently, if $p(t) = (At + B)g^{-1}(t)$, where $A$ and $B$ are
constants, then \eqref{tran} is of the form \eqref{id}, since
$D_{q,\omega}(p(t)g(t))$ is constant. Thus, the function
\begin{equation*}
y(t)=(At+C)g^{-1}(t)
\end{equation*}
with
\begin{equation*}
A=\left[\alpha g(a)-\beta g(b)\right](a-b)^{-1}, \quad
C=\left[a\beta g(b)-b\alpha g(a)\right](a-b)^{-1},
\end{equation*}
minimizes \eqref{illust:Ex:4}.
\end{ex}

Using the idea of Leitmann, we can also solve quantum optimal
control problems defined in terms of Hahn's operators.

\begin{ex}
\label{example3}
Let $q$, $\omega$ be real numbers on
a closed interval $I$ of $\mathbb{R}$ such that
$\omega_0\in I$ and
$0,1 \in \{q^ns+[n]_{q,\omega} : n\in\mathbb{N}_{0}\}\cup\{\omega_0\}$
for some $s \in I$.
Consider the global minimum problem
\begin{equation}
\label{ex1:a} \text{minimize} \quad \mathcal{L}[u_1,u_2]
= \int_0^1 \left((u_1(t))^2 + u_2(t))^2\right) d_{q,\omega} t\\
\end{equation}
subject to the control system
\begin{equation}
\label{ex1:b}
D_{q,\omega}y_1(t) = \exp(u_1(t)) + u_1(t) + u_2(t) \, ,
\quad
D_{q,\omega}y_2(t) = u_2(t) \, ,
\end{equation}
and conditions
\begin{equation}
\label{ex1:c}
y_1(0) = 0 \, , \quad y_1(1) = 2 \, , \quad
y_2(0) = 0 \, , \quad y_2(1) = 1 \, , \quad
u_1(t) \, , u_2(t) \in \Omega = [-1,1] \, .
\end{equation}
This example is inspired from \cite{withLeitmann}. It is worth to
mention that due to the constraints on the values of the controls
($u_1(t)$, $u_2(t) \in \Omega = [-1,1]$), a theory based
on necessary optimality conditions to solve problem
\eqref{ex1:a}--\eqref{ex1:c} does not exist at the moment.

We begin noticing that problem \eqref{ex1:a}--\eqref{ex1:c} is
variationally invariant according to \cite{GouveiaTorres05} under
the one-parameter family of transformations
\begin{equation}
\label{transfEx1} y_1^s = y_1 + s t \, , \quad y_2^s = y_2 + s t \,
, \quad u_2^s = u_2 + s \quad (t^s = t \text{ and } u_1^s = u_1) \,
.
\end{equation}
To prove this, we need to show that both the functional integral
$\mathcal{L}$ and the control system stay invariant under the
$s$-parameter transformations \eqref{transfEx1}. This is easily seen
by direct calculations:
\begin{equation}
\label{inv:Func:Ex1}
\begin{split}
\mathcal{L}^s[u_1^s,u_2^s]&=
\int_0^1 \left(u_1^{s}(t)\right)^2 + \left(u_2^{s}(t)\right)^2  d_{q,\omega} t\\
&= \int_0^1 u_1(t)^2 + \left(u_2(t) + s \right)^2 d_{q,\omega} t \\
&= \int_0^1 \left( u_1(t)^2 + u_2(t)^2
+  D_{q,\omega} t[s^2 t + 2 s y_2(t)]\right) d_{q,\omega} t\\
&= \mathcal{L}[u_1,u_2] + s^2 + 2s \, .
\end{split}
\end{equation}
We remark that $\mathcal{L}^s$ and $\mathcal{L}$ have the same
minimizers: adding a constant $s^2 + 2s$ to the functional
$\mathcal{L}$ does not change the minimizer of $\mathcal{L}$. It
remains to prove that the control system also remains invariant
under transformations \eqref{transfEx1}:
\begin{equation}
\label{inv:CS:Ex1}
\begin{split}
D_{q,\omega}\left(y_1^{s}(t)\right) &= D_{q,\omega}\left(y_1(t) + s
t \right)
= D_{q,\omega}y_1 + s = \exp(u_1(t)) + u_1(t) + u_2(t) + s \\
&= \exp(u_1^s(t)) + u_1^s(t) + u_2^s(t) \, , \\
D_{q,\omega}\left(y_2^{s}(t)\right) &= D_{q,\omega}\left(y_2(t) +
st\right)
= D_{q,\omega}y_2 + s = u_2(t) + s \\
&= u_2^{s}(t) \, .
\end{split}
\end{equation}
Conditions \eqref{inv:Func:Ex1} and \eqref{inv:CS:Ex1} prove that
problem \eqref{ex1:a}--\eqref{ex1:c} is invariant under the
$s$-parameter transformations \eqref{transfEx1} up to
$D_{q,\omega}\left(s^2 t + 2s y_2(t)\right)$. Using the invariance
transformations \eqref{transfEx1}, we generalize problem
\eqref{ex1:a}--\eqref{ex1:c} to a $s$-parameter family of problems,
$s \in \mathbb{R}$, which include the original problem for $s=0$:
\begin{equation*}
\text{minimize} \quad \mathcal{L}^{s}[u_1,u_2] = \int_0^1
(u_1^s(t))^2 + (u_2^s(t))^2 d_{q,\omega}t
\end{equation*}
subject to the control system
\begin{equation*}
D_{q,\omega}\left(y_1^s(t)\right) = \exp(u_1^s(t)) + u_1^s(t) + u_2^s(t) \, ,\quad
D_{q,\omega}\left(y_2^s(t)\right) = u_2^s(t) \, ,
\end{equation*}
and conditions
\begin{equation*}
\begin{gathered}
y_1^s(0) = 0 \, , \quad y_1^s(1) = 2 + s \, , \quad
y_2^s(0) = 0 \, , \quad y_2^s(1) = 1 + s \, , \\
u_1^s(t) \in [-1,1] \, ,  \quad u_2^s(t) \in [-1+s,1+s] \, .
\end{gathered}
\end{equation*}
It is clear that $ \mathcal{L}^{s}\geq 0$ and that $\mathcal{L}^s=0$
if $u_1^{s}(t)=u_2^{s}(t) \equiv 0$. The control equations, the
boundary conditions and the constraints on the values of the
controls imply that $u_1^{s}(t)=u_2^{s}(t) \equiv 0$ is admissible
only if $s=-1$: $y_1^{s=-1}(t) = t$, $y_2^{s=-1}(t) \equiv 0$.
Hence, for $s= -1$ the global minimum to $\mathcal{L}^s$ is 0 and
the minimizing trajectory is given by
\begin{equation*}
\tilde{u}_1^{s}(t) \equiv 0  \, , \quad \tilde{u}_2^{s}(t) \equiv 0
\, , \quad \tilde{y}_1^{s}(t)= t  \, , \quad \tilde{y}_2^{s}(t)
\equiv 0 \, .
\end{equation*}
Since for any $s$ one has by \eqref{inv:Func:Ex1} that
$\mathcal{L}[u_1,u_2] = \mathcal{L}^s[u_1^s,u_2^s] - s^2 - 2s$, we
conclude that the global minimum for problem $\mathcal{L}[u_1,u_2]$
is 1. Thus, using the inverse functions of the variational
symmetries \eqref{transfEx1},
\begin{equation*}
u_1(t) = u_1^{s}(t) \, , \quad u_2(t) = u_2^{s}(t)- s \, , \quad
y_1(t) = y_1^{s}(t) - s t \, , \quad y_2(t) = y_2^{s}(t) - s t \, ,
\end{equation*}
and the absolute minimizer for problem \eqref{ex1:a}--\eqref{ex1:c}
is
\begin{equation*}
\tilde{u}_1(t) = 0 \, , \quad \tilde{u}_2(t) = 1 \, , \quad
\tilde{y}_1(t) = 2 t \, , \quad \tilde{y}_2(t) = t \, .
\end{equation*}
\end{ex}

%----------------------------------------------------------

\subsection{An Application Towards a Quantum Ramsey Model}
\label{app}

As the variables, that are usually considered and observed by the
economist, are the outcome of a great number of decisions, taken by
different operators at different points of time, it seems natural to
look for new kinds of models which are more flexible and
realistic. Hahn's approach allows for more complex applications
than the discrete or the continuous models. A consumer might have
income from work at unequal time intervals and/or make expenditures
at unequal time intervals. Therefore, it is possible to obtain
more rigorous and more accurate solutions with the approach here proposed.

We discuss the application of the Hahn quantum variational calculus
to the Ramsey model, which determines the behavior of
saving/consumption as the result of optimal inter-temporal choices
by individual households \cite{Atici}. For a complete treatment of
the classical Ramsey model we refer the reader to \cite{barro}.
Before writing the quantum model in terms of the Hahn operators we
will present its discrete and continuous versions. The discrete-time
Ramsey model is
\begin{equation*}
\max_{[W_t]} \quad
\sum_{t=0}^{T-1}(1+p)^{-t}U\left[W_t-\frac{W_{t+1}}{1+r}\right],
\quad C_t=W_t-\frac{W_{t+1}}{1+r},
\end{equation*}
while the continuous Ramsey model is
\begin{equation}
\label{c:RM} \max_{W(\cdot)} \quad
\int_{0}^{T}e^{-pt}U\left[rW(t)-W'(t)\right]dt, \quad
C(t)=rW(t)-W'(t),
\end{equation}
where the quantities are defined as
\begin{itemize}
\item $W$ -- production function,
\item $C$ -- consumption,
\item $p$ -- discount rate,
\item $r$ -- rate of yield,
\item $U$ -- instantaneous utility function.
\end{itemize}

One may assume, due to some constraints of economical nature, that
the dynamics do not depend on the usual derivative or the forward
difference operator, but on the Hahn quantum difference operator
$D_{q,\omega}$. In this condition, one is entitled to assume again
that the constraint $C(t)$ has the form
\begin{equation*}
C(t)=-\left[E\left(-r,\frac{t-\omega}{q}\right)\right]^{-1}
D_{q,\omega}\left[E\left(-r,\frac{t-\omega}{q}\right)W(t)\right],
\end{equation*}
where $E\left(z,\cdot\right)$ is the $q,\omega$-exponential function
defined by
\begin{equation*}
E\left(z,t\right):=\prod _{k=0}^{\infty}(1+zq^k(t(1-q)-\omega))
\end{equation*}
for $z\in \mathbb{C}$. Several nice properties of the
$q,\omega$-exponential function can be found in
\cite{Aldwoah,Annaby}. By taking the $q,\omega$-derivative of
$\left[E\left(-r,\frac{t-\omega}{q}\right)W(t)\right]$ the following
is obtained:
\begin{multline*}
C(t)=-\left[E\left(-r,\frac{t-\omega}{q}\right)\right]^{-1}\left[
E\left(-r,\frac{t-\omega}{q}\right)D_{q,\omega}W(t)\right.\\
\left.+E\left(-r,\frac{t-\omega}{q}\right)W(qt+\omega)
\frac{r\left(1-\frac{1}{q}\right)-r\left(1+r\left(t
-\frac{t-\omega}{q}\right)\right)}{\left(1+r\left(t
-\frac{t-\omega}{q}\right)\right)\left(1-r\left(t(1-q)
-\omega\right)\right)}\right].
\end{multline*}
The quantum Ramsey model with the Hahn difference operator consists to
\begin{equation}
\label{q:RM} \max_{W(\cdot)} \quad
\int_{0}^{T}E(-p,t)U\left[W(qt+\omega)
\frac{r\left(1+r\left(t-\frac{t-\omega}{q}\right)\right)
-r\left(1-\frac{1}{q}\right)}{\left(1+r\left(t-\frac{t-\omega}{q}\right)
\right)\left(1-r\left(t(1-q)-\omega\right)\right)}-D_{q,\omega}W(t)\right]d_{q,\omega}
\end{equation}
subject to the constraint
\begin{equation}
\label{eq:const:qRM} C(t)= W(qt+\omega)
\frac{r\left(1+r\left(t-\frac{t-\omega}{q}\right)\right)
-r\left(1-\frac{1}{q}\right)}{\left(1+r\left(t-\frac{t-\omega}{q}\right)
\right)\left(1-r\left(t(1-q)-\omega\right)\right)}-D_{q,\omega}W(t).
\end{equation}
The quantum Euler--Lagrange equation is, by Theorem~\ref{thm:mr},
given by
\begin{equation}
\label{q:EL:RM}
E(-p,t)U'\left[C(t)\right]\frac{r\left(1+r\left(t-\frac{t-\omega}{q}\right)\right)
-r\left(1-\frac{1}{q}\right)}{\left(1+r\left(t-\frac{t-\omega}{q}\right)
\right)\left(1-r\left(t(1-q)-\omega\right)\right)}
+D_{q,\omega}\left[E(-p,t)U'\left[C(t)\right]\right]=0.
\end{equation}
Note that for $q\uparrow1$ and $\omega \downarrow 0$ problem
\eqref{q:RM}--\eqref{eq:const:qRM} reduces to \eqref{c:RM}, and
\eqref{q:EL:RM} to the classical Ramsey's Euler-Lagrange
differential equation.

% -------------------------------------------------------

\section{Conclusion}
\label{sec:conc}

In this paper we consider variational problems
in the context of the Hahn quantum calculus.
Such variational problems are defined through the Hahn
quantum difference operator and the Jackson--N\"orlund integral.
The origin of the Hahn quantum difference operator dates back to a 1949
paper of W. Hahn \cite{Hahn} where it was introduced to unify, in a
limiting sense, the Jackson $q$-difference derivative and the forward difference.
For both of these latter two quantum difference operators,
variational problems have been studied previously. The forward
difference problems were studied at least as early as 1937
by T. Fort \cite{Fort1937} and for the $q$-difference by G. Bangerezako in 2004
\cite{Bang04}. In both of these works the authors discuss necessary
conditions for optimality and obtain the analogue of the classical Euler--Lagrange
equation, as well as other classical results. The goal of the present paper is to
provide extensions of the previous results for the more general
Hahn quantum difference operator.

Another related course of study is that of the notion
of a time scale. The origins of this idea dates back to the late 1980's when
S.~Hilger introduced this notion in his Ph.D. thesis (directed by B.~Aulbach)
and showed how to unify continuous time and discrete time dynamical
systems \cite{Hilger}. Since this important result, the literature has exploded with papers
and books on time scales in which many known results for ordinary differential
equations and difference equations have been combined and extended \cite{mal:tor,B:04,Atici}.
The classical results of the calculus of variations have been extended to times scales
by M. Bohner in 2004 \cite{B:04}. However, the Euler-Lagrange equation here obtained
is not comparable with that of \cite{B:04}.
Indeed, the Hahn quantum calculus is not covered
by the Hilger time scale theory. This is well explained, for example,
in the 2009 Ph.D. thesis of Aldwoah \cite{Aldwoah} (see also \cite{Hamza}).
Here we just note the following: if in Bohner's paper \cite{B:04}
one chooses the time scale to be the $q$-scale
$\mathbb{T} := \{ q^n : n \in \mathbb{Z}\}$,
then the expression of the delta-derivative coincides with the expression
of the Jackson $q$-difference derivative. However,
they are not the same. There is an important distinction:
the Jackson $q$-difference derivative is defined in the set of real numbers
while the time-scale derivative is only defined in a subset $\mathbb{T}$
of the real numbers. One more difference, between the Hahn calculus we use
in this paper and the time scale theory, is the following: the delta integral
satisfies all the usual properties of the Riemann integral while
this is not the case with the Jackson--N\"orlund integral:
the inequality \eqref{surp:ineq}
is not always true for the Jackson--N\"orlund integral.

The main advantage of our results is that they are able
to deal with nondifferentiable functions,
even discontinuous functions, that are important in physical systems.
Quantum derivatives and integrals play a leading role in the understanding of complex
physical systems. For example, in 1992 Nottale introduced the theory of scale-relativity
without the hypothesis of space-time differentiability \cite{Nottale}.
A rigorous mathematical foundation to Nottale's scale-relativity theory
is nowadays given by means of a quantum calculus \cite{Almeida,Cresson}.
We remark that results in Bohner's paper \cite{B:04} are not able to deal
with such nondifferentiable functions. Variational problems in \cite{B:04}
are formulated for functions that are delta-differentiable.
It is well known that delta-differentiable functions are necessarily continuous.
This is not the case in our context: see Example~\ref{ex:disctF:withD},
where a discontinuous function is $q,\omega$-differentiable
in all the real interval $[-1,1]$.

We believe that the obtained results are of interest in Economics.
Economists model time as continuous or discrete. The kind of
``time'' (continuous or discrete) to be used in the
construction of dynamic models is a moot question. Although
individual economic decisions are generally made at discrete time
intervals, it is difficult to believe that they are perfectly
synchronized as postulated by discrete models. The usual assumption that
the economic activity takes place continuously, is a convenient
abstraction in many applications. In others, such as the ones studied in
financial-market equilibrium, the assumption of continuous trading
corresponds closely to reality. We believe that our Hahn's approach
helps to bridge the gap between two families of models: continuous
and discrete. We trust that the field here initiated will prove fruitful
for further research.

% -------------------------------------------------------

% ---------------------------------------------------


\begin{thebibliography}{99}

\bibitem{Almeida}
Almeida, R., Torres, D.F.M.:
H\"olderian variational problems subject to integral constraints.
J. Math. Anal. Appl. {\bf 359}(2), 674--681 (2009)
{\tt arXiv:0807.3076}

\bibitem{Bang04}
Bangerezako, G.:
Variational $q$-calculus.
J. Math. Anal. Appl. {\bf 289}(2), 650--665 (2004)

\bibitem{Bangerezako}
Bangerezako, G.:
Variational calculus on $q$-nonuniform lattices.
J. Math. Anal. Appl. {\bf 306}(1), 161--179 (2005)

\bibitem{Cresson}
Cresson, J., Frederico, G.S.F., Torres, D.F.M.:
Constants of motion for non-differentiable quantum variational problems.
Topol. Methods Nonlinear Anal. {\bf 33}(2), 217--231 (2009)
{\tt arXiv:0805.0720}

\bibitem{Kac}
Kac, V., Cheung, P.:
{\it Quantum calculus}. Springer, New York (2002)

\bibitem{Hahn}
Hahn, W.: \"Uber orthogonalpolynome, die $q$-differenzenlgleichungen gen\"ugen.
Math. Nachr. {\bf 2}, 4--34 (1949)

\bibitem{Gasper}
Gasper, G., Rahman, M.:
{\it Basic hypergeometric series}, Second edition.
Cambridge Univ. Press, Cambridge (2004)

\bibitem{Jackson1}
Jackson, F.H.:
Basic integration.
Quart. J. Math., Oxford Ser. (2) {\bf 2}, 1--16 (1951)

\bibitem{Bird}
Bird, M.T.:
On generalizations of sum formulas of the Euler-Maclaurin type.
Amer. J. Math. {\bf 58}(3), 487--503 (1936)

\bibitem{Jagerman}
Jagerman, D.L.:
{\it Difference equations with applications to queues}.
Dekker, New York (2000)

\bibitem{Jordan}
Jordan, C.:
{\it Calculus of finite differences}, Third Edition.
Introduction by Carver, H.C. Chelsea, New York (1965).

\bibitem{alvares}
\'Alvarez-Nodarse, R.:
On characterizations of classical polynomials,
J. Comput. Appl. Math. {\bf 196}(1), 320--337 (2006)

\bibitem{costas}
Costas-Santos, R.S., Marcell\'an, F.:
Second structure relation for $q$-semiclassical polynomials of the Hahn tableau.
J. Math. Anal. Appl., {\bf 329}(1), 206--228 (2007)

\bibitem{odzi}
Dobrogowska, A., Odzijewicz, A.:
Second order $q$-difference equations solvable by factorization method.
J. Comput. Appl. Math. {\bf 193}(1), 319--346 (2006)

\bibitem{Kwon}
Kwon, K.H., Lee, D.W., Park, S.B., Yoo, B.H.:
Hahn class orthogonal polynomials.
Kyungpook Math. J. {\bf 38}(2), 259--281 (1998)

\bibitem{Petronilho}
Petronilho, J.:
Generic formulas for the values at the singular
points of some special monic classical
$H_{q,\omega }$-orthogonal polynomials.
J. Comput. Appl. Math. {\bf 205}(1), 314--324 (2007)

\bibitem{Aldwoah}
Aldwoah, K.A.: Generalized time scales and associated difference
equations. Ph.D. Thesis, Cairo University (2009)

\bibitem{Annaby}
Annaby, M.H., Hamza, A.E., Aldwoah, K.A.:
Hahn difference operator and associated Jackson-N\"orlund integrals.
Preprint (2009)

\bibitem{Kelley:Peterson}
Kelley, W.G., Peterson, A.C.:
{\it Difference equations}, Second edition.
Harcourt/Academic Press, San Diego, CA (2001)

\bibitem{Fort1937}
Fort, T.:
The calculus of variations applied to N\"orlund's sum,
Bull. Amer. Math. Soc. {\bf 43}(12), 885--887 (1937)

\bibitem{Leitmann67}
Leitmann, G.:
A note on absolute extrema of certain integrals.
Internat. J. Non-Linear Mech. {\bf 2}, 55--59 (1967)

\bibitem{Car}
Carlson, D.A.:
An observation on two methods of obtaining solutions
to variational problems.
J. Optim. Theory Appl. {\bf 114}(2), 345--361 (2002)

\bibitem{CarlsonLeitmann05a}
Carlson, D.A., Leitmann, G.:
Coordinate transformation method
for the extremization of multiple integrals.
J. Optim. Theory Appl. {\bf 127}(3), 523--533 (2005)

\bibitem{CarlsonLeitmann05b}
Carlson, D.A., Leitmann, G.:
A direct method for open-loop dynamic games for affine control systems.
In: {\it Dynamic games: theory and applications}, pp.~37--55.
Springer, New York (2005)

\bibitem{TE}
Carlson, D.A., Leitmann, G.:
Fields of extremals and sufficient conditions
for the simplest problem of the calculus of variations.
J. Global Optim. {\bf 40}(1-3), 41--50 (2008)

\bibitem{Leit}
Leitmann, G.:
On a class of direct optimization problems.
J. Optim. Theory Appl. {\bf 108}(3), 467--481 (2001)

\bibitem{Leitmann01}
Leitmann, G.:
Some extensions to a direct optimization method.
J. Optim. Theory Appl. {\bf 111}(1), 1--6 (2001)

\bibitem{MR1954118}
Leitmann, G.:
On a method of direct optimization.
Vychisl. Tekhnol. {\bf 7}, 63--67 (2002)

\bibitem{MR2065731}
Leitmann, G.:
A direct method of optimization and its application
to a class of differential games.
Cubo Mat. Educ. {\bf 5}(3), 219--228 (2003)

\bibitem{MR2035262}
Leitmann, G.:
A direct method of optimization and its application
to a class of differential games.
Dyn. Contin. Discrete Impuls.
Syst. Ser. A Math. Anal. {\bf 11}(2-3), 191--204 (2004)

\bibitem{mal:tor}
Malinowska, A.B., Torres, D.F.M.:
Leitmann's direct method of optimization
for absolute extrema of certain problems
of the calculus of variations on time scales.
Appl. Math. Comput. (2010), in press.
DOI: 10.1016/j.amc.2010.01.015
{\tt arXiv:1001.1455}

\bibitem{SilvaTorres06}
Silva, C.J., Torres, D.F.M.:
Absolute extrema of invariant optimal control problems.
Commun. Appl. Anal. {\bf 10}(4), 503--515 (2006)
{\tt arXiv:math/0608381}

\bibitem{withLeitmann}
Torres, D.F.M., Leitmann, G.:
Contrasting two transformation-based methods
for obtaining absolute extrema.
J. Optim. Theory Appl. {\bf 137}(1), 53--59 (2008)
{\tt arXiv:0704.0473}

\bibitem{Wagener}
Wagener, F.O.O.:
On the Leitmann equivalent problem approach.
J. Optim. Theory Appl. {\bf 142}(1), 229--242 (2009)

\bibitem{Hamza}
Aldwoah, K.A., Hamza, A.E.:
Difference time scales,
Int. J. Math. Stat. {\bf 9}(A11), 106--125 (2011)

\bibitem{Jackson2}
Jackson, F.H.:
On $q$-definite integrals.
Quart. J. Pure and Appl. Math. {\bf 41}, 193--203 (1910)

\bibitem{Fort}
Fort, T.:
{\it Finite Differences and Difference Equations in the
Real Domain}. Oxford, at the Clarendon Press (1948)

\bibitem{nol}
N\"orlund, N.:
Vorlesungen \"uber Differencenrechnung.
Springer Verlag, Berlin (1924)

\bibitem{B:04}
Bohner, M.:
Calculus of variations on time scales.
Dynam. Systems Appl. {\bf 13}(3-4), 339--349 (2004)

\bibitem{Brunt}
van Brunt, B.:
{\it The calculus of variations}.
Springer, New York (2004)

\bibitem{Arutyunov}
Arutyunov, A.V.:
{\it Optimality conditions---Abnormal and degenerate problems}.
Kluwer Acad. Publ., Dordrecht (2000)

\bibitem{Torres:EJC}
Torres, D.F.M.:
On the Noether theorem for optimal control.
Eur. J. Control {\bf 8}(1), 56--63 (2002)

\bibitem{Torres:CPAA}
Torres, D.F.M.:
Proper extensions of Noether's symmetry theorem
for nonsmooth extremals of the calculus of variations.
Commun. Pure Appl. Anal. {\bf 3}(3), 491--500 (2004)

\bibitem{Torres:JMS}
Torres, D.F.M.:
Carath\'eodory equivalence Noether theorems,
and Tonelli full-regularity in the calculus
of variations and optimal control.
J. Math. Sci. (N. Y.) {\bf 120}(1), 1032--1050 (2004)
{\tt arXiv:math/0206230}

\bibitem{GouveiaTorres05}
Gouveia, P.D.F., Torres, D.F.M.:
Automatic computation of conservation laws
in the calculus of variations and optimal control.
Comput. Methods Appl. Math. {\bf 5}(4), 387--409 (2005)
{\tt arXiv:math/0509140}

\bibitem{GouveiaTorresRocha}
Gouveia, P.D.F., Torres, D.F.M., Rocha, E.A.M.:
Symbolic computation of variational symmetries in optimal control.
Control Cybernet. {\bf 35}(4), 831--849 (2006)
{\tt arXiv:math/0604072}

\bibitem{Atici}
Atici, F.M., McMahan, C.S.:
A comparison in the theory of calculus of variations
on time scales with an application to the Ramsey Model.
Nonlinear Dyn. Syst. Theory {\bf 9}(1), 1--10 (2009)

\bibitem{barro}
Barro, R.J., Sala-i-Martin, X.:
Economic growth. MIT Press,
Cambridge (1999)

\bibitem{Hilger}
Aulbach, B., Hilger, S.:
A unified approach to continuous and discrete dynamics,
in {\it Qualitative theory of differential equations (Szeged, 1988)},
37--56, Colloq. Math. Soc. J\'anos Bolyai, 53, North-Holland, Amsterdam (1990)

\bibitem{Nottale}
Nottale, L.: The theory of scale relativity,
Internat. J. Modern Phys. A {\bf 7}(20), 4899--4936 (1992)

\end{thebibliography}
\end{document}